\documentclass[12pt,reqno]{amsart}

\usepackage{graphicx,subfigure}
\usepackage{hyperref}
\hypersetup{colorlinks=true, citecolor=blue, linkcolor=red}

\numberwithin{equation}{section} \numberwithin{figure}{section}
\numberwithin{table}{section} 
{ \theoremstyle{plain}

}

{ \theoremstyle{definition}
\newtheorem{thm}{Theorem}

\numberwithin{equation}{section} \numberwithin{lem}{section}
\numberwithin{thm}{section} \numberwithin{cor}{section}
\numberwithin{pro}{section} \numberwithin{rem}{section}
}

\begin{document}

\title{On a maximum principle for vector minimizers to the Allen-Cahn energy}



\author{Christos Sourdis} \address{Department of Mathematics and Applied Mathematics, University of
Crete.}
              \email{csourdis@tem.uoc.gr}           




\maketitle

\begin{abstract}
By using the unique continuation principle for linear elliptic
systems, we can simplify the proof of a recent variational maximum
principle due to Alikakos and Fusco \cite{alikakosPreprint}. At the
same time, this approach allows us to relax an assumption from the
latter reference.
\end{abstract}


In the recent paper \cite{alikakosPreprint}, the authors proved the
following variational maximum principle, which has already found
several applications (see \cite{alikakosPreprint},
\cite{sourdis14}):

\begin{thm}\label{thmAF}
Let $W:\mathbb{R}^m\to \mathbb{R}$ be $C^1$ and nonnegative. Assume
that $W(a)=0$ for some $a\in \mathbb{R}^m$ and that there is $r_0>0$
such that for $\nu \in \mathbb{S}^{m-1}$ the map
\begin{equation}\label{eqfunctions}
r\to W(a+r\nu )\ \ r\in (0,r_0],
\end{equation}
has strictly positive derivative. Let $A\subset \mathbb{R}^n$ be an
open, connected, bounded set, with $\partial A$ minimally smooth,
and suppose that $u\in W^{1,2}(A;\mathbb{R}^m)\cap
L^\infty(A;\mathbb{R}^m)$ minimizes
\[
J_A(v)=\int_{A}^{}\left(\frac{1}{2}|\nabla v|^2+W(v)\right)dx
\]
subject to its Dirichlet values $v=u$ on $\partial A$.

If there holds
\[
|u(x)-a|\leq r\ \ \textrm{for}\ x\in \partial A,
\]
for some $r\in (0,r_0/2)$, then it also holds that
\[
|u(x)-a|\leq r\ \ \textrm{for}\ x\in  A.
\]
\end{thm}
The main idea of the proof is that if the assertion is violated at
some point, then one can construct a suitable  competitor function
which agrees with $u$ on $\partial A$ and has strictly less energy,
which is impossible.

In this note, under some slight additional regularity on $W$ (which
is consistent with applications to the corresponding elliptic
system), we will show that one can conclude just by showing that the
aforementioned competitor function has less or equal energy. Our
main observation is to apply the unique continuation principle for
linear elliptic systems (see \cite{lopez} for other applications).
As a result, we can simplify the corresponding proof in
\cite{alikakosPreprint} and also allow for the functions in
(\ref{eqfunctions}) to be merely nondecreasing. More precisely, we
have the following theorem.

\begin{thm}
Assume that $W:\mathbb{R}^m\to \mathbb{R}$ is $C^{1,1}$,
nonnegative, such that $W(a)=0$ for some $a\in \mathbb{R}^m$ and
that the functions in (\ref{eqfunctions}) are nondecreasing.
Moreover, assume that
\[
W(u)>0 \ \ \textrm{if}\ \ |u-a|<2r_0\ \ \textrm{and}\ \ x\neq a.
\]
Then, the assertion of   Theorem \ref{thmAF} remains true.
\end{thm}
\begin{proof}
Firstly, by standard elliptic regularity theory, we have that $u$ is
smooth in $A$ and continuous up to the boundary (under reasonable
assumptions on $\partial A$). Without loss of generality, we take
$a=0$. As in \cite{alikakosPreprint}, we set
\[
\rho(x)=|u(x)|\ \ \textrm{and}\ \ \nu(x)=\frac{u(x)}{\rho(x)}
\]
on $A_+=\{x\in A\ :\ \rho>0 \}$. We also set $A_0=\{x\in A\ :\
\rho=0 \}$. It has been shown in \cite{alikakosPreprint} that the
energy of $u$ equals
\[
J_A(u)=\frac{1}{2}\int_{A}^{}|\nabla
\rho|^2dx+\frac{1}{2}\int_{A_+}^{}\rho^2 |\nabla \nu|^2
dx+\int_{A}^{}W(\rho \nu)dx.
\]
Let
\[
\tilde{u}(x)=\left\{\begin{array}{ll}
                      \textrm{min}\left\{\rho(x),r \right\}\alpha\left(\rho(x) \right)\nu(x), & x\in A_+\cap \{\rho<2r \}, \\
  &   \\
                      0, & x\in A_0\cup \{\rho \geq 2r \},
                    \end{array}
 \right.
\]
where $\alpha(\cdot)$ is the auxiliary function
\[
\alpha(\tau)=\left\{\begin{array}{ll}
                      1, & \tau\leq r, \\
                        &   \\
                      \frac{2r-\tau }{r},  & r\leq \tau \leq 2 r, \\
                        &   \\
                      0, & \tau \geq 2 r.
                    \end{array}
 \right.
\]
It was shown in \cite{alikakosPreprint} that $\tilde{u}\in
W^{1,2}(A;\mathbb{R}^m)\cap L^\infty(A;\mathbb{R}^m)$ and that its
energy equals
\[
J_A(\tilde{u})=\frac{1}{2}\int_{A}^{}|\nabla
\tilde{\rho}|^2dx+\frac{1}{2}\int_{\tilde{A}_+}^{}\tilde{\rho}^2
|\nabla \nu|^2 dx+\int_{A}^{}W(\tilde{\rho} \nu)dx,
\]
where $\tilde{\rho}(x)=|\tilde{u}(x)|$ and $\tilde{A}_+=\{x\in A\ :\
\tilde{\rho}>0 \}$. Note that
\begin{equation}\label{eqcontra}
u=\tilde{u}\ \ \textrm{on}\ \ \partial A\ \ \textrm{and}\ \
|\tilde{u}|\leq r\ \ \textrm{a.e.\ in}\ \ A.
\end{equation}
It follows readily that
\[
J_A(\tilde{u})\leq J_A(u),
\]
see also the proof in \cite{alikakosPreprint}. Consequently,
$\tilde{u}$ is also a minimizer subject to the same boundary
conditions as $u$. It follows that $\tilde{u}$ is smooth and
satisfies
\[
\Delta \tilde{u}=\nabla W(\tilde{u})\ \ \textrm{in}\ \ A.
\]

Suppose, to the contrary, that
\begin{equation}\label{eqcontra2}
|u(x_0)|>r\ \ \textrm{for some}\ \ x_0\in A.
\end{equation}
We will first exclude the case
\[
r\leq \rho(x)\leq 2r\ \ \textrm{for all}\ \ x\in A.
\]
If not, the function
\[
\hat{u}=r\nu(x)\in W^{1,2}(A;\mathbb{R}^m)\cap
L^\infty(A;\mathbb{R}^m)
\]
would have strictly less energy then $u$ (because
$\int_{A}^{}|\nabla \rho|^2dx>0$) while $\hat{u}=u$ on $\partial A$,
which is impossible. Next, we exclude entirely the case
\[
r\leq \rho(x),\ \ x\in A.
\]
If not, there would exist $x_1\in A$ such that $\rho(x_1)>2r$. This
implies that $\tilde{u}=0$ on a set of positive measure containing
$x_1$. Since $\nabla W(u)$ is locally Lipschitz continuous, we see
that $\tilde{u}$ satisfies the linear system
\[
\Delta \tilde{u}=Q(x)\tilde{u}\ \ \textrm{in}\ \ A, \ \
\textrm{where} \ \ Q(x)=\int_{0}^{1}\partial^2W(t u)u dt\ \
\textrm{is bounded in norm}.
\]
On the other hand, because $\tilde{u}=0$ on a set of positive
measure, by the unique continuation principle for linear elliptic
systems (see \cite{hormander}), we infer that $\tilde{u}\equiv 0$
which is clearly impossible (otherwise $|u|\geq 2r$ in $A$).
Therefore, we may assume that there exists a set $B\subset A$ with
positive measure such that
\[
u=\tilde{u}\ \ \textrm{in}\ \ B.
\]
As before, by considering the linear system for the difference
$u-\tilde{u}$, we conclude that $\tilde{u}\equiv u$. We have thus
arrived at a contradiction, because of (\ref{eqcontra}) and
(\ref{eqcontra2}).
\end{proof}

\end{document}